\documentclass{amsart}
\usepackage{amsmath}
\usepackage{amssymb}
\usepackage{amsfonts}
\usepackage{graphicx}
\usepackage{tikz}
\usepackage{longtable}
\usepackage{array}

\usepackage{lipsum} 
\usepackage[T1]{fontenc}
\newtheorem{theorem}{Theorem}

\newcommand{\qsymbol}[6]{\left|\begin{array}{lll}
    #1 & #2 & #3 \\
    #4 & #5 & #6 
    \end{array}\right|}
    
\newcommand{\qshsymbol}[6]{\left|\begin{array}{lll}
    #1 & #2 & #3 \\
    #4 & #5 & #6 
    \end{array}\right|'}
    
\newcommand{\bsymbol}[6]{\left(\begin{array}{lll}
    #1 & #2 & #3 \\
    #4 & #5 & #6 
    \end{array}\right)}

\title{Homologically trivial part of the Turaev -- Viro invariant order 7}
\author{Korablev Ph. G.}

\address{Chelyabinsk State University, Chelyabinsk, Russia; N.N. Krasovsky Institute of Mathematics and Meckhanics, Ekaterinburg, Russia}
\email{korablev@csu.ru}
\date{}

\begin{document}

\begin{abstract}
    A homologically trivial part of any Turaev -- Viro invariant odd order $r$ is a Turaev -- Viro type invariant order $\frac{r + 1}{2}$. In this paper we find an explicit formulas for this Turaev -- Viro type invariant, corresponding to the invariant order $r = 7$. Our formulas express $6j$-symbols and colour weights in terms of $\gamma$, where $\gamma$ is a root of the polynomial $\mathcal{T}(x) = x^3 - 2x^2 - x + 1$.
\end{abstract}

\maketitle

\section{Introduction}

Turaev -- Viro invariants for 3-manifolds were proposed by V. Turaev and O. Viro in 1992 in the paper \cite{TuraevViro}. They constructed an infinite family of invariants, each of which is determined by the choice of its order $r\geqslant 2$. Let $TV_r$ denote the Turaev -- Viro invariant of order $r$.

For each $r\geqslant 2$ define the set $C_r = \{0, 1, \ldots, r - 2\}$. It is convenient to think of this set $C_r$ as the set of colours used to colour 3-manifold spine cells when computing the values of the invariant $TV_r$. For the purposes of this article, the topological meaning of the elements of this set is not important. One way to define the invariant $TV_r$ is to fix the complex values $w_i$ for all $i\in C_r$, called colour weights, and so-called $6j$-symbols, denoted by $\qsymbol{i}{j}{k}{l}{m}{n}$, for all $i, j, k, l, m, n\in C_r$.

The values of all $6j$-symbols and the colour weights that define the $TV_r$ invariant are not arbitrary. They are described by formulas which are given, for example, in the books \cite[section 8.1.4]{Matveev} and \cite[section XII.8.5]{Turaev}, article \cite[section 7]{TuraevViro} and below in the second paragraph. Quantum integers and quantum factorials play a central role in these formulas. Let $q$ be a root of unity of degree $2r$ such that $q^2$ is a primitive root of unity of degree $r$. Let us denote $$[n]_r = \dfrac{q^n - q^{-n}}{q - q^{-1}}.$$ The value $[n]_r$ is called a quantum integer. This value is always real, but it depends on the choice of the root $q$. Obviously $[1]_r = 1$ and $[r]_r = 0$. The quantum factorial is the value $$[n]_r! = [n]_r\cdot [n - 1]_r\cdot \ldots\cdot [1]_r.$$

There are invariants of 3-manifolds given by other values of $6j$-symbols and colour weights. According to the terminology proposed in the book \cite[remark 8.1.20]{Matveev}, such invariants are called Turaev -- Viro type invariants. To construct a Turaev -- Viro type invariant of order $r$, it suffices to choose values of the colour weights $w_i$ and $6j$-symbols such that they satisfy the equations $$\qsymbol{i}{j}{k}{l}{ m}{n}\cdot \qsymbol{i}{j}{k}{l'}{m'}{n'} = \sum\limits_{z\in C_r}w_z\cdot \qsymbol{i}{ m}{n}{z}{n'}{m'}\cdot \qsymbol{j}{l}{n}{z}{n'}{m'}\cdot \qsymbol{k}{l} {m}{z}{m'}{l'}$$ for all $i, j, k, l, m, n, l', m', n' \in C_r$.

An example of a non-trivial Turaev -- Viro type invariant of order 3 is the $\varepsilon$-invariant (see \cite{TInvariant}). It is given by the following set of colour weights and $6j$-symbols (see \cite[section 8.1.2]{Matveev}): $w_0 = 1$, $w_1 = \varepsilon$ and 

\begin{center}
{\setlength{\tabcolsep}{0.25cm}
\begin{tabular}{lll}
$\qsymbol{0}{0}{0}{0}{0}{0} = 1$, & $\qsymbol{0}{0}{0}{1}{1}{1} = \dfrac{1}{\sqrt{\varepsilon}}$, & $\qsymbol{0}{1}{1}{0}{1}{1} = \dfrac{1}{\varepsilon}$, \\[0.35cm]
$\qsymbol{0}{1}{1}{1}{1}{1} = \dfrac{1}{\varepsilon}$, & $\qsymbol{1}{1}{1}{1}{1}{1} = -\dfrac{1}{\varepsilon^2}$,
\end{tabular}
}
\end{center}
where $\varepsilon$ is a root of the equation $x^2 - x - 1 = 0$. All other $6j$-symbols are zero.

It is well known that each Turaev -- Viro invariant $TV_r$ of order $r$ is the sum of three invariants $TV_{r, 0}$, $TV_{r, 1}$ and $TV_{r, 2}$ (see \cite{ThreeParts}). The invariant $TV_{r, 0}$ is called the homologically trivial part of the Turaev -- Viro invariant. For odd $r$ it is a Turaev -- Viro type invariant of order $\frac{r + 1}{2}$. Let's denote it $TH_{\frac{r + 1}{2}}$.

Let us denote the weights of the colours defining the invariant $TH_{\frac{r + 1}{2}}$ by $w_i'$, $i\in C_{\frac{r + 1}{2}} = \{0, 1, \ldots, \frac{r - 3}{2}\}$, and $6j$-symbols by $\qshsymbol{i}{j}{k}{l}{m}{n}$, $i, j, k, l, m, n\in C_{\frac{r + 1}{2}}$. There is a simple way to find the weights of the colours and $6j$-symbols that define the $TH_{\frac{r + 1}{2}}$ invariant, knowing the weights of the colours and $6j$-symbols that define the $TV_r$ invariant:
\begin{center}
    $w_i' = w_{2i}$ for all $i\in C_{\frac{r + 1}{2}}$,
    
    $\qshsymbol{i}{j}{k}{l}{m}{n} = \qsymbol{2i}{2j}{2k}{2l}{2m}{2n}$ for all $i, j, k, l, m, n\in C_{\frac{r + 1}{2}}$.
\end{center}

It can be shown that the $\varepsilon$-invariant coincides with the $TH_{3}$ invariant, i.e. it is the homologically trivial part of the Turaev -- Viro invariant of order 5 (see \cite[theorem 8.1.26]{Matveev} and \cite{TInvariant}).

The purpose of this article is to find formulas that define the $TH_{4}$ invariant, i.e. the homologically trivial part of the $TV_7$ invariant, similar to formulas that define the $\varepsilon$-invariant. The main result of the article is formulated in the theorem \ref{Theorem:GammaInvariant}. In this theorem the values of the colour weights and $6j$-symbols that define the $TH_{4}$ invariant are expressed by the root $\gamma$ of the equation $x^3 - 2x^2 - x + 1 = 0$. This description is completely similar to the description of the $\varepsilon$-invariant by the root $\varepsilon$ of the equation $x^2 - x - 1 = 0$. Therefore, the $TH_{4}$ invariant is called the $\gamma$-invariant.

The structure of the article is as follows. In the second section we write explicit formulas for colour weights and $6j$-symbols, which define the invariant $TV_{7, 0}$. These formulas are written in terms of quantum integers $[n]_7$. The third section is devoted to the simplest properties of the polynomial $\mathcal{T}(x) = x^3 - 2x^2 - x + 1$, which plays a central role in the construction of the $\gamma$-invariant. In the fourth section, the main theorem \ref{Theorem:GammaInvariant} is proved. In the fifth section we compute polynomials corresponding to the homologically trivial parts of the Turaev -- Viro invariants of odd orders $5\leqslant r\leqslant 21$.

\section{Turaev -- Viro invariant order 7}

As before, denote $C_r = \{0, \ldots, r - 2\}$. The books \cite[section 8.1.4]{Matveev}, \cite[section XII.8.5]{Turaev} and the article \cite[section 7]{TuraevViro} contain explicit formulas for calculating the values of the weights $w_i$ and $6j$-symbols $\qsymbol{i}{j}{k}{l}{m}{n}$, $i, j, k, l, m, n \in C_r$, which define the Turaev-Viro invariant of order $ r\geqslant 2$. Let's present these formulas.

\begin{center}
    $w_i = (-1)^i [i + 1]_r$ for all $i\in C_r$.
\end{center}

Let us say that the triplet $(x, y, z)\in C_r^3$ is admissible if and only if
\begin{enumerate}
\item $x + y\geqslant z$, $y + z\geqslant x$ and $z + x\geqslant y$;
\item $x + y + z$ is even;
\item $x + y + z \leqslant 2\cdot r - 4$.
\end{enumerate}

The value of the $6j$-symbol $\qsymbol{i}{j}{k}{l}{m}{n}$ is 0 if at least one of the triples $(i, j, k)$, $(k , l, m)$, $(m, n, i)$, $(j, l, n)$ is not admissible. The values of the remaining $6j$-symbols are calculated as follows
$$\qsymbol{i}{j}{k}{l}{m}{n} = \sum\limits_{z = \alpha}^{\beta}\frac{(-1)^z\cdot [z + 1]_r!\cdot A(i, j, k, l, m, n)}{B\left(z, \bsymbol{i}{j}{k}{l}{m}{n}\right)\cdot C\left(z, \bsymbol{i}{j}{k}{l}{m}{n}\right)},$$
where 
\begin{center}
$A(i, j, k, l, m, n) = \emph{\textbf{i}}^{i+j+k+l+m+n}\cdot \Delta(i, j, k)\cdot \Delta(i, m, n)\cdot \Delta(j, l, n)\cdot \Delta(k, l, m),$
\end{center}
\begin{center}
$\emph{\textbf{i}}\in\mathbb{C}$ --- imaginary unit,
\end{center}
$$B\left(z, \bsymbol{i}{j}{k}{l}{m}{n}\right) = [z - \widehat{i} - \widehat{j} - \widehat{k}]_r!\cdot [z - \widehat{i} - \widehat{m} - \widehat{n}]_r! \cdot [z - \widehat{j} - \widehat{l} - \widehat{n}]_r!\cdot [z - \widehat{k} - \widehat{l} - \widehat{m}]_r!,$$
$$C\left(z, \bsymbol{i}{j}{k}{l}{m}{n}\right) = [\widehat{i} + \widehat{j} + \widehat{l} + \widehat{m} - z]_r!\cdot [\widehat{i} + \widehat{k} + \widehat{l} + \widehat{n} - z]_r!\cdot [\widehat{j} + \widehat{k} + \widehat{m} + \widehat{n}] - z]_r!,$$
$$\Delta(i, j, k) = \sqrt{\frac{[\widehat{i} + \widehat{j} - \widehat{k}]_r!\cdot [\widehat{j} + \widehat{k} - \widehat{i}]_r!\cdot [\widehat{k} + \widehat{j} - \widehat{i}]_r!}{[\widehat{i} + \widehat{j} + \widehat{k} + 1]_r!}},$$

\begin{center}
$\widehat{x} = \frac{x}{2}$ for all $x\in C_r$,

$\alpha = \max(\widehat{i} + \widehat{j} + \widehat{k}, \widehat{i} + \widehat{m} + \widehat{n}, \widehat{j} + \widehat{l} + \widehat{n}, \widehat{k} + \widehat{l} + \widehat{m})$,

$\beta = \min(\widehat{i} + \widehat{j} + \widehat{l} + \widehat{m}, \widehat{i} + \widehat{k} + \widehat{l} + \widehat{n}, \widehat{j} + \widehat{k} + \widehat{m} + \widehat{n})$.
\end{center}

Since in this article we are mainly talking about the Turaev -- Viro invariant of order 7, we will write $[n]$ instead of $[n]_7$ for simplicity.

\begin{theorem}
    \label{Theorem:Symbols}
    The homologically trivial part $TV_{7, 0}$ of the Turaev -- Viro invariant of order 7 is defined by the following values of the weights $w_i$, $i\in\{0, 2, 4\}$ and $6j$-symbols $\qsymbol{i}{j}{k}{l}{m}{n}$, $i, j, k, l, m, n\in\{0, 2, 4\}$: $w_0 = 1$, $w_2 = [3]$, $w_4 = [5]$ and
    
    {\setlength{\tabcolsep}{0.75cm}
    \begin{longtable}{ll}
    $\qsymbol{0}{0}{0}{0}{0}{0} = 1$ & $\qsymbol{0}{0}{0}{2}{2}{2} = -\dfrac{1}{\sqrt{[3]}}$ \\[0.35cm]
    $\qsymbol{0}{0}{0}{4}{4}{4} = \dfrac{1}{\sqrt{[5]}}$ &
    $\qsymbol{0}{2}{2}{0}{2}{2} = \dfrac{1}{[3]}$ \\[0.35cm]
    $\qsymbol{0}{2}{2}{2}{2}{2} = \dfrac{1}{[3]}$ & $\qsymbol{0}{2}{2}{2}{4}{4} = -\dfrac{1}{\sqrt{[3]\cdot [5]}}$ \\[0.35cm]
    $\qsymbol{0}{2}{2}{4}{2}{2} = \dfrac{1}{[3]}$ & $\qsymbol{0}{2}{2}{4}{4}{4} = -\dfrac{1}{\sqrt{[3]\cdot [5]}}$ \\[0.35cm]
    $\qsymbol{0}{4}{4}{0}{4}{4} = \dfrac{1}{[5]}$ &
    $\qsymbol{0}{4}{4}{2}{4}{4} = \dfrac{1}{[5]}$ \\[0.35cm]
    $\qsymbol{2}{2}{2}{2}{2}{2} = \dfrac{[5] - 1}{[2]\cdot [3]\cdot [4]}$ & $\qsymbol{2}{2}{2}{2}{2}{4} = -\dfrac{[2]}{[3]\cdot [4]}$ \\[0.35cm]
    $\qsymbol{2}{2}{2}{2}{4}{4} = -\dfrac{1}{[4]}\cdot\sqrt{\dfrac{[2]\cdot [6]}{[3]\cdot [5]}}$ & $\qsymbol{2}{2}{2}{4}{4}{4} = \dfrac{[3]}{[4]\cdot\sqrt{[2]\cdot [3]\cdot [5]\cdot [6]}}$ \\[0.35cm]
    $\qsymbol{2}{2}{4}{2}{2}{4} = \dfrac{[2]}{[3]\cdot [4]\cdot [5]}$ &
    $\qsymbol{2}{2}{4}{2}{4}{4} = \dfrac{[2]}{[4]\cdot [5]}$ \\[0.35cm]
    $\qsymbol{2}{4}{4}{2}{4}{4} = -\dfrac{1}{[4]\cdot [5]\cdot [6]}$
    \end{longtable}
    }
    
    All other $6j$-symbols and weights $w_i$, $i\in\{1, 3, 5\}$ are zero.
\end{theorem}
\begin{proof}
    The statement of the theorem is obtained by a careful calculation using the formulas given above. Since we are only interested in the homologically trivial part of the invariant, all weights $w_i$ for odd $i\in C_r$ are zero. The values of all $6j$-symbols $\qsymbol{i}{j}{k}{l}{m}{n}$ are zero if at least one of the numbers $i, j, k, l, m, n\in C_7$ is odd.
\end{proof}

\section{Polynomial $\mathcal{T}(x) = x^3 - 2x^2 - x + 1$}

A simple analysis shows that the polynomial $\mathcal{T}(x)$ has three real roots: one is greater than 2, let's denote it $\gamma_1$, the second belongs to the interval $(0, 1)$, let's denote it $\gamma_2$, and the third root is negative, let's denote it $\gamma_3$ (see figure \ref{Figure:TRoots}).

\begin{figure}[h!]
    \begin{tikzpicture}[scale=1.5]
    \draw[->, thick] (-1.5, 0.0) -- (3.5, 0.0);
    \filldraw (-1.0, 0.0) node[below] {$-1$} circle (0.02);
    \filldraw (0.0, 0.0) node[below] {$0$} circle (0.02);
    \filldraw (1.0, 0.0) node[below] {$1$} circle (0.02);
    \filldraw (2.0, 0.0) node[below] {$2$} circle (0.02);
    \filldraw (3.0, 0.0) node[below] {$3$} circle (0.02);
    
    \filldraw (-0.802, 0.0) node[above] {$\gamma_3$} circle (0.03);
    \filldraw (0.555, 0.0) node[above] {$\gamma_2$} circle (0.03);
    \filldraw (2.247, 0.0) node[above] {$\gamma_1$} circle (0.03);
    \end{tikzpicture}
    \caption{\label{Figure:TRoots}Roots of the polynomial $\mathcal{T}(x) = x^3 - 2x^2 - x + 1$.}
\end{figure}
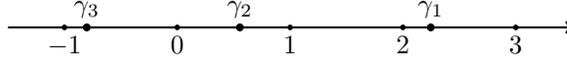

\begin{theorem}
    Let $\gamma$ be a root of the polynomial $\mathcal{T}(x)$. Then $1 - \frac{1}{\gamma}$ is also a root of the polynomial $\mathcal{T}(x)$.
\end{theorem}
\begin{proof}
    First, note that $\gamma\neq 0$. Furthermore, since $$\gamma^3 - 2\gamma^2 - \gamma + 1 = 0,$$ dividing both sides of this equality by $\gamma^3$, we get $$1 - \frac{2}{\gamma} - \frac{1}{\gamma^2} + \frac{1}{\gamma^3} = 0.$$
    
    Calculate $\mathcal{T}(1 - \frac{1}{\gamma})$.
    
    \begin{multline*}
    \mathcal{T}\left(1 - \frac{1}{\gamma}\right) = \left(1 - \frac{1}{\gamma}\right)^3 - 2\left(1 - \frac{1}{\gamma}\right)^2 - \left(1 - \frac{1}{\gamma}\right) + 1 = \\ = -1 + \frac{2}{\gamma} + \frac{1}{\gamma^2} - \frac{1}{\gamma^3} = -\left(1 - \frac{2}{\gamma} - \frac{1}{\gamma^2} + \frac{1}{\gamma^3}\right) = 0.
    \end{multline*}
\end{proof}

Consider the function $\tau(x)\colon\mathbb{R}\to\mathbb{R}$, which acts as follows: $$\tau(x) = 1 - \frac{1}{x}.$$ It's easy to see that $\gamma_2 = \tau(\gamma_1)$, $\gamma_3 = \tau(\gamma_2)$ and $\gamma_1 = \tau(\gamma_3)$. The last equality follows from the fact that the function $\tau$ has order 3, i.e. $\tau\circ\tau\circ\tau = id$.

\section{Turaev -- Viro type invariant $TH_4$}

\begin{theorem}
    \label{Theorem:Three}
    $\mathcal{T}([3]) = 0$.
\end{theorem}
\begin{proof}
    It follows from the definition that $[3] = q^2 + 1 + q^{-2}$, where $q\neq \pm 1$ is a root of unity of degree 14. Let's calculate $\mathcal{T}(q^2 + 1 + q^{-2})$ explicitly:

    \begin{multline*}
    \mathcal{T}(q^2 + 1 + q^{-2}) = (q^2 + 1 + q^{-2})^3 - 2(q^2 + 1 + q^{-2})^2 - (q^2 + 1 + q^{-2}) + 1 = \\ = q^6 + q^4 + q^2 + 1 + q^{-2} + q^{- 4} + q^{-6}.
    \end{multline*}

    Note that if $\omega = q^2$, then $\omega$ is a primitive root of unity of degree 7. Then $$\mathcal{T}([3]) = \omega^6 + \omega^5 + \omega^4 + \omega^3 + \omega^2 + \omega + 1 = 0.$$
\end{proof}

When calculating the value $[3]$, we can take $q$ to be any root of unity of degree 14 (except $\pm 1$). Figure \ref{Figure:Roots} shows different roots of the polynomial $\mathcal{T}(x)$ corresponding to different choices of $q$. As before, $\gamma_1$ is the largest root of the polynomial, $\gamma_2$ is the middle, and $\gamma_3$ is the smallest.

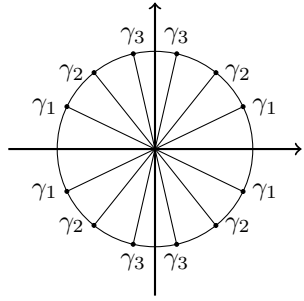
\begin{figure}[h!]
    \begin{center}
    \begin{tikzpicture}[scale=1.3]
    \draw[->, thick] (-1.5, 0.0) -- (1.5, 0.0);
    \draw[->, thick] (0.0, -1.5) -- (0.0, 1.5);
    \draw (0, 0) circle (1.0);
    
    \draw (0.0, 0.0) -- (25.71: 1.0);
    \filldraw (25.71: 1.0) node[right] {$\gamma_1$} circle (0.02);
    
    \draw (0.0, 0.0) -- (25.71 * 2: 1.0);
    \filldraw (25.71 * 2: 1.0) node[right] {$\gamma_2$} circle (0.02);
    
    \draw (0.0, 0.0) -- (25.71 * 3: 1.0);
    \filldraw (25.71 * 3: 1.0) node[above] {$\gamma_3$} circle (0.02);
    
    \draw (0.0, 0.0) -- (25.71 * 4: 1.0);
    \filldraw (25.71 * 4: 1.0) node[above] {$\gamma_3$} circle (0.02);
    
    \draw (0.0, 0.0) -- (25.71 * 5: 1.0);
    \filldraw (25.71 * 5: 1.0) node[left] {$\gamma_2$} circle (0.02);
    
    \draw (0.0, 0.0) -- (25.71 * 6: 1.0);
    \filldraw (25.71 * 6: 1.0) node[left] {$\gamma_1$} circle (0.02);
    
    \draw (0.0, 0.0) -- (25.71 * 8: 1.0);
    \filldraw (25.71 * 8: 1.0) node[left] {$\gamma_1$} circle (0.02);
    
    \draw (0.0, 0.0) -- (25.71 * 9: 1.0);
    \filldraw (25.71 * 9: 1.0) node[left] {$\gamma_2$} circle (0.02);
    
    \draw (0.0, 0.0) -- (25.71 * 10: 1.0);
    \filldraw (25.71 * 10: 1.0) node[below] {$\gamma_3$} circle (0.02);
    
    \draw (0.0, 0.0) -- (25.71 * 11: 1.0);
    \filldraw (25.71 * 11: 1.0) node[below] {$\gamma_3$} circle (0.02);
    
    \draw (0.0, 0.0) -- (25.71 * 12: 1.0);
    \filldraw (25.71 * 12: 1.0) node[right] {$\gamma_2$} circle (0.02);
    
    \draw (0.0, 0.0) -- (25.71 * 13: 1.0);
    \filldraw (25.71 * 13: 1.0) node[right] {$\gamma_1$} circle (0.02);
    \end{tikzpicture}
    \end{center}
    \caption{\label{Figure:Roots}Different values of the 14-th root of the unit and the corresponding root of the polynomial $\mathcal{T}(x)$, which is equal to $[3]$.}
\end{figure}

\begin{theorem}
    \label{Theorem:Identities}
    The following identities hold:
    \begin{enumerate}
        \item $[5]\cdot [3] = [3] + [5]$;
        \item $[2]\cdot [4] = [3]\cdot [5]$;
        \item $[2]\cdot [3] = [4]\cdot [5]$;
        \item $[2]\cdot [6] = [5]$;
        \item $[4]\cdot [6] = [3]$;
        \item $[3]^2 = [4]^2$.
    \end{enumerate}
\end{theorem}
\begin{proof}
    Each relation is proved by substituting
    \begin{center}
        $[2] = q + q^{-1}$,
        
        $[3] = q^2 + 1 + q^{-2}$,
        
        $[4] = q^3 + q + q^{-1} + q^{-3}$,
        
        $[5] = q^4 + q^2 + 1 + q^{-2} + q^{-4}$,
        
        $[6] = q^5 + q^3 + q + q^{-1} + q^{-3} + q^{-5}$
    \end{center}
    into the required equations, opening brackets, and bringing similar. The only additional relation used is that if $\omega = q^2$, then $$1 + \omega + \omega^2 + \omega^3 + \omega^4 + \omega^5 + \omega^6 = 0.$$
\end{proof}

From the first statement of the theorem \ref{Theorem:Identities} it follows that if $[3] = \gamma$ is an arbitrary root of the polynomial $\mathcal{T}(x)$, then $$[5] = \dfrac{\gamma}{\gamma - 1}.$$

The homologically trivial part of any Turaev -- Viro invariant of order $r$ for odd $r$ is a Turaev -- Viro type invariant of order $\frac{r + 1}{2}$. To distinguish $6j$-symbols of this Turaev -- Viro type invariant from the original invariant of order $r$, we denote them by $\qshsymbol{i}{j}{k}{l}{m}{n}$ and weights $w_i'$.

\begin{theorem}
    \label{Theorem:GammaInvariant}
    Let $\gamma$ be an arbitrary root of the polynomial $\mathcal{T}(x)$. Then the invariant $TH_{4}$ of the Turaev -- Viro type of order 4, which coincides with the homologically trivial part $TV_{7, 0}$ of the Turaev -- Viro invariant of order 7, is given by the following values of weights $w_i'\in\{0, 1, 2\}$ and $6j$-symbols $\qshsymbol{i}{j}{k}{l}{m}{n}$, $i, j, k, l, m, n\in\{0, 1, 2\}$: $w_0' = 1$, $w_1' = \gamma$, $w_2' = \frac{\gamma}{\gamma - 1}$ and

    {\setlength{\tabcolsep}{1cm}
    \begin{longtable}{ll}
        $\qshsymbol{0}{0}{0}{0}{0}{0} = 1$ & $\qshsymbol{0}{0}{0}{1}{1}{1} = -\dfrac{1}{\sqrt{\gamma}}$ \\[0.35cm]
        $\qshsymbol{0}{0}{0}{2}{2}{2} = \dfrac{\sqrt{\gamma - 1}}{\sqrt{\gamma}}$ &
        $\qshsymbol{0}{1}{1}{0}{1}{1} = \dfrac{1}{\gamma}$ \\[0.35cm]
        $\qshsymbol{0}{1}{1}{1}{1}{1} = \dfrac{1}{\gamma}$ & $\qshsymbol{0}{1}{1}{1}{2}{2} = -\dfrac{\sqrt{\gamma - 1}}{\gamma}$ \\[0.35cm]
        $\qshsymbol{0}{1}{1}{2}{1}{1} = \dfrac{1}{\gamma}$ & $\qshsymbol{0}{1}{1}{2}{2}{2} = -\dfrac{\sqrt{\gamma - 1}}{\gamma}$ \\[0.35cm]
        $\qshsymbol{0}{2}{2}{0}{2}{2} = \dfrac{\gamma - 1}{\gamma}$ & $\qshsymbol{0}{2}{2}{1}{2}{2} = \dfrac{\gamma - 1}{\gamma}$ \\[0.35cm]
        $\qshsymbol{1}{1}{1}{1}{1}{1} = \dfrac{1}{\gamma^3}$ & $\qshsymbol{1}{1}{1}{1}{1}{2} = -\dfrac{1}{\gamma\cdot (\gamma - 1)}$ \\[0.35cm]
        $\qshsymbol{1}{1}{1}{1}{2}{2} = -\dfrac{1}{\gamma\cdot\sqrt{\gamma}}$ & $\qshsymbol{1}{1}{1}{2}{2}{2} = \dfrac{\gamma - 1}{\gamma\cdot\sqrt{\gamma}}$ \\[0.35cm]
        $\qshsymbol{1}{1}{2}{1}{1}{2} = \dfrac{1}{\gamma^2}$ &
        $\qshsymbol{1}{1}{2}{1}{2}{2} = \dfrac{1}{\gamma}$ \\[0.35cm]
        $\qshsymbol{1}{2}{2}{1}{2}{2} = -\dfrac{\gamma - 1}{\gamma^2}$
    \end{longtable}}

    All other $6j$-symbols are zero.
\end{theorem}
\begin{proof}
    Let the value of the root $q$ be chosen so that $[3] = \gamma$. Let us calculate the values of the required colour weights and $6j$-symbols, using the theorems \ref{Theorem:Symbols} and \ref{Theorem:Identities}, and the fact that
    \begin{center}
        $[3] = \gamma$ and $[5] = \dfrac{\gamma}{\gamma - 1}$.
    \end{center}

    First calculate the colour weights $w_i'$, $i\in\{0, 1, 2\}$.

    $w_0' = w_0 = 1$, $w_1' = w_2 = [3] = \gamma$ and $w_2' = w_4 = [5] = \dfrac{\gamma}{\gamma - 1}$.

    Next values of $6j$-symbols.

    $\qshsymbol{0}{0}{0}{0}{0}{0} = \qsymbol{0}{0}{0}{0}{0}{0} = 1$.

    $\qshsymbol{0}{0}{0}{1}{1}{1} = \qsymbol{0}{0}{0}{2}{2}{2} = -\dfrac{1}{\sqrt{[3]}} = -\dfrac{1}{\sqrt{\gamma}}$.

    $\qshsymbol{0}{0}{0}{2}{2}{2} = \qsymbol{0}{0}{0}{4}{4}{4} = \dfrac{1}{\sqrt{[5]}} = \dfrac{\sqrt{\gamma - 1}}{\sqrt{\gamma}}$.

    $\qshsymbol{0}{1}{1}{0}{1}{1} = \qsymbol{0}{2}{2}{0}{2}{2} = \dfrac{1}{[3]} = \dfrac{1}{\sqrt{\gamma}}$.

    $\qshsymbol{0}{1}{1}{1}{1}{1} = \qsymbol{0}{2}{2}{2}{2}{2} = \dfrac{1}{[3]} = \dfrac{1}{\sqrt{\gamma}}$.

    $\qshsymbol{0}{1}{1}{1}{2}{2} = \qsymbol{0}{2}{2}{2}{4}{4} = -\dfrac{1}{\sqrt{[3]\cdot [5]}} = -\dfrac{\sqrt{\gamma - 1}}{\gamma}$.

    $\qshsymbol{0}{1}{1}{2}{1}{1} = \qsymbol{0}{2}{2}{4}{2}{2} = \dfrac{1}{[3]} = \dfrac{1}{\sqrt{\gamma}}$.

    $\qshsymbol{0}{1}{1}{2}{2}{2} = \qsymbol{0}{2}{2}{4}{4}{4} = -\dfrac{1}{\sqrt{[3]\cdot [5]}} = -\dfrac{\sqrt{\gamma - 1}}{\gamma}$ .

    $\qshsymbol{0}{2}{2}{0}{2}{2} = \qsymbol{0}{4}{4}{0}{4}{4} = \dfrac{1}{[5]} = \dfrac{\gamma - 1}{\gamma}$.

    $\qshsymbol{0}{2}{2}{1}{2}{2} = \qsymbol{0}{4}{4}{2}{4}{4} = \dfrac{1}{[5]} = \dfrac{\gamma - 1}{\gamma}$.

    $\qshsymbol{1}{1}{1}{1}{1}{1} = \qsymbol{2}{2}{2}{2}{2}{2} = \dfrac{[5] - 1}{[2]\cdot [3]\cdot [4]}$. Use the second statement of the theorem \ref{Theorem:Identities} and replace in the expression $[2]\cdot [4]$ with $[3]\cdot [5]$. As a result we get the value $\dfrac{[5] - 1}{[3]^2\cdot [5]} = \dfrac{1}{\gamma^3}$.

    $\qshsymbol{1}{1}{1}{1}{1}{2} = \qsymbol{2}{2}{2}{2}{2}{4} = -\dfrac{[2]}{[3]\cdot [4]}$. Use the third statement of the theorem \ref{Theorem:Identities} and replace $\dfrac{[2]}{[4]}$ with $\dfrac{[5]}{[3]}$. We get the value $-\dfrac{[5]}{[3]^2} = -\dfrac{1}{\gamma\cdot (\gamma - 1)}$.

    $\qshsymbol{1}{1}{1}{1}{2}{2} = \qsymbol{2}{2}{2}{2}{4}{4} = -\dfrac{1}{ [4]}\cdot\sqrt{\dfrac{[2]\cdot [6]}{[3]\cdot [5]}}$. Use the fourth relation from the theorem \ref{Theorem:Identities} and replace $[2]\cdot [6]$ with $[5]$. We get the expression $-\dfrac{1}{[4]\cdot\sqrt{[3]}}$. From the sixth relation of the theorem \ref{Theorem:Identities} it follows that $[4] = \pm\gamma$, but for different choices of $q$ the value of $[4]$ can coincide with both $\gamma$ and $-\gamma$. As noted in \cite[remarks 8.1.17 and 8.1.18]{Matveev}, the correctness and the value of the invariant does not depend on the choice of sign when extracting the square root. Therefore, for definiteness, we can choose $[4] = \gamma$, and then the value of the  $6j$-symbol is equal to $-\dfrac{1}{\gamma\cdot\sqrt{\gamma}}$.

    $\qshsymbol{1}{1}{1}{2}{2}{2} = \qsymbol{2}{2}{2}{4}{4}{4} = \dfrac{[3]}{[4]\cdot\sqrt{[2]\cdot [3]\cdot [5]\cdot [6]}}$. Use the fourth relation from the theorem \ref{Theorem:Identities} and replace $[2]\cdot [6]$ with $[5]$ and, similar to the previous case, choose for definiteness $[4] = \gamma$. We get the value $\dfrac{\gamma - 1}{\gamma\cdot\sqrt{\gamma}}$.

    $\qshsymbol{1}{1}{2}{1}{1}{2} = \qsymbol{2}{2}{4}{2}{2}{4} = \dfrac{[2]}{[3]\cdot [4]\cdot [5]}$. Use the third statement of the theorem \ref{Theorem:Identities} and replace $\dfrac{[2]}{[4]}$ with $\dfrac{[5]}{[3]}$. We get the value $\dfrac{1}{[3]^2} = \dfrac{1}{\gamma^2}$.

    $\qshsymbol{1}{1}{2}{1}{2}{2} = \qsymbol{2}{2}{4}{2}{4}{4} = \dfrac{[2]}{[4]\cdot [5]}$. Use the third statement of the theorem \ref{Theorem:Identities} and replace $\dfrac{[2]}{[4]}$ with $\dfrac{[5]}{[3]}$. We get the value $\dfrac{1}{[3]} = \dfrac{1}{\gamma}$.

    $\qshsymbol{1}{2}{2}{1}{2}{2} = \qsymbol{2}{4}{4}{2}{4}{4} = -\dfrac{1}{[4]\cdot [5]\cdot [6]}$. Use the fifth statement of the theorem \ref{Theorem:Identities} and replace $[4]\cdot [6]$ with $[3]$. We get the value $-\dfrac{1}{[3]\cdot [5]} = -\dfrac{\gamma - 1}{\gamma}$.
\end{proof}

\section{Polynomials for homologically trivial parts of Turaev -- Viro invariants}

As already noted, the homologically trivial part of the Turaev -- Viro invariant of odd order $r$ is a Turaev -- Viro type invariant of order $\frac{r + 1}{2}$. For $r = 5$ the corresponding invariant of order 3 is the $\varepsilon$-invariant, and for $r = 7$ it is the $\gamma$-invariant described in the previous paragraph. Both invariants are expressed in terms of the roots of two special polynomials. For the $\varepsilon$-invariant this polynomial is equal to $x^2 - x - 1$ and for the $\gamma$-invariant it is equal to $x^3 - 2x^2 - x + 1$.

It is natural to assume that any Turaev -- Viro type invariant $TH_{\frac{r + 1}{2}}$ can be expressed in terms of the roots of a suitable polynomial. This is consistent with the fact that the values of the Turaev -- Viro invariants are algebraic integers (\cite{AlgebraicNumbers}). Looking at the formulas defining the Turaev -- Viro invariants of order $r$, it is difficult to guess which polynomials they should be. The procedure for choosing them is based on the observation proved in Theorem \ref{Theorem:Three}. To do this, we need to compute the values of $[3]_r$ for all possible $q \neq\pm 1$, which are the roots of unity of degree $2r$. Then we need to construct a polynomial whose roots are all the different computed values. The table \ref{Table:Polynomials} lists the first few polynomials constructed using this rule. The first column of the table contains the order of the Turaev -- Viro type invariant, which coincides with the homologically trivial part of the Turaev -- Viro invariant $TV_{r, 0}$ of odd order $5\leqslant r\leqslant 21$. The second column gives the corresponding polynomial.

\setcounter{table}{0}

\begin{table}[h!]
    \begin{tabular}{|c|l|} \hline
    Order & Polynomial \\ \hline
    $3$ & $x^2 - x - 1$ \\
    $4$ & $x^3 - 2x^2 - x + 1$ \\
    $5$ & $x^4 - 3x^3 + 3 x$ \\
    $6$ & $x^5 - 4x^4 + 2x^3 + 5x^2 - 2x - 1$ \\
    $7$ & $x^6 - 5x^5 + 5x^4 + 6x^3 - 7x^2 - 2x + 1$ \\
    $8$ & $x^7 - 6x^6 + 9x^5 + 5x^4 - 15x^3 + 5 x$ \\
    $9$ & $x^8 - 7x^7 + 14x^6 + x^5 - 25x^4 + 9x^3 + 12x^2 - 3x - 1$ \\
    $10$ & $x^9 - 8x^8 + 20x^7 - 7x^6 - 35 x^5 + 29 x^4 + 18 x^3 - 15 x^2 - 3x + 1$ \\
    $11$ & $x^{10} - 9x^9 + 27x^8 - 20 x^7 - 42 x^6 + 63 x^5 + 14 x^4 - 42 x^3 + 7x$ \\ \hline
    \end{tabular}
    \caption{\label{Table:Polynomials}Polynomials for invariants $TH_{r}$.}
\end{table}

We do not formulate or prove any explicit statements about the polynomials in the table \ref{Table:Polynomials}. The procedure described above should be regarded as an empirical rule. It can be used to construct required polynomials.


\bigskip
\begin{thebibliography}{1}
\bibitem{TuraevViro} V.G. Turaev, O.Ya. Viro, {\it State sum invariants of 3-manifolds and quantum 6j-symbols}, Topology, \textbf{31}:4 (1992), 865--902.

\bibitem{Matveev} S. Matveev, {\it Algorithmic Topology and Classification of 3-Manifolds}, Springer Science \& Business Media, 2007.

\bibitem{Turaev} V.G. Turaev, {\it Quantum invariants of knots and 3-manifolds}, de Gruyter Studies in Mathematics, vol. 18, Walter de Gruyter, Berlin, 1994.

\bibitem{TInvariant} S.V. Matveev, M.V. Sokolov, {\it On a simple invariant of Turaev-Viro type}, Journal of Mathematical Sciences, \textbf{94} (1999), 1226-–1229.

\bibitem{ThreeParts} M.V. Sokolov, {\it The Turaev-Viro invariant for 3-manifolds is a sum of three invariants}, Canad. Math. Bull., \textbf{39}:4 (1996), 468--475.

\bibitem{AlgebraicNumbers} G. Masbaum, J.D. Roberts, {\it A simple proof of integrality of quantum invariants at prime roots of unity}, Mathematical Proceedings of the Cambridge Philosophical Society, \textbf{121}:3 (1997), 443--454.
\end{thebibliography}
\end{document}